\documentclass[12pt]{article}
\usepackage{amsthm}
\usepackage{fullpage}
\usepackage{amsmath}
\usepackage{mathrsfs}
\usepackage{amssymb}
\usepackage{enumerate}
\usepackage{array}
\usepackage[pdfstartview=FitH,backref,colorlinks]{hyperref}
\usepackage{graphicx}
\def\qed{\hfill \ifhmode\unskip\nobreak\fi\quad\ifmmode\Box\else$\Box$\fi\\ }
\usepackage{subfigure}

\newcommand{\beq}[1]{\begin{equation}\label{#1}}
\newcommand{\eeq}{\end{equation}}
\newcommand{\blem}[1]{\begin{lemma}\label{#1}}
\newcommand{\elem}{\end{lemma}}
\newcommand{\bth}[1]{\begin{theorem}\label{#1}}
\newcommand{\enth}{\end{theorem}}
\newcommand{\brem}[1]{\begin{remark}\label{#1}}
\newcommand{\erem}{\end{remark}}

\usepackage[colorinlistoftodos,prependcaption,textsize=scriptsize, textwidth=23mm,]{todonotes}

\newcommand{\xtodo}[1]{\todo[color=yellow]{X: #1}}

\newtheorem{theorem}              {Theorem}
\newtheorem{lemma}      [theorem] {Lemma}

\newtheorem{corollary}      [theorem] {Corollary}
\newtheorem{remark}      [theorem] {Remark}

\newtheorem{claim}{Claim}[section]

\newcommand{\bc}{{\bf c}}

\usepackage{color}

\begin{document}

\author{Alexandr Kostochka\thanks{
Department of Mathematics, University of Illinois at Urbana-Champaign, Urbana, IL 61801, USA.
{\tt kostochk@illinois.edu.} The research of this author is partially supported by NSF grant   DMS-2153507 and NSF RTG grant DMS-1937241.
}
\and {Jingwei Xu}\thanks{
Department of Mathematics, University of Illinois at Urbana-Champaign, Urbana, IL 61801, USA.
{\tt jx6@illinois.edu.}}
\and Xuding Zhu\thanks{Department of Mathematics, Zhejiang Normal University,  China.  E-mail: {\tt xdzhu@zjnu.edu.cn}. The research of this author is partially supported by
Grants  NSFC 12371359, U20A2068}}

\title{Sparse critical graphs for defective $(1,3)$-coloring}

\maketitle

\begin{abstract}
A graph $G$ is $(1,3)$-colorable if its vertices can be partitioned into subsets $V_1$ and $V_2$ so
that every vertex in $G[V_1]$ has degree at most $1$
and every vertex in $G[V_2]$ has degree at most $3$.
We prove that every graph
with maximum average degree at most 28/9
 is $(1, 3)$-colorable.  
\end{abstract}

\section{Introduction}

A partition of the vertex set of a graph $G$ into $k$
 subsets $V_1,\ldots, V_k$ is a {\em defective $(d_1,\ldots,d_k)$-coloring of $G$} (or, simply a
{\em $(d_1,\ldots,d_k)$-coloring) of $G$} if
 for each $i\in[k]$, every vertex in $V_i$ has at most $d_i$ neighbors in $V_i$. 
 This notion generalizes those of proper $k$-coloring (when $d_1 =\ldots = d_k = 0$) and of $d$-improper $k$-coloring (when $d_1 =\ldots = d_k = d$).

Probably, the first result on $d$-defective coloring for $d > 0$ is due to Gerencs\'er~\cite{Ge}: 
he showed that each graph $G$ with
maximum degree $D$ has a $1$-defective $(\lfloor D/2\rfloor+1)$-coloring. This was generalized by
Lov\' asz~\cite{LL66}, who proved  that if
$(d_1+1)+(d_2+1)+\ldots +(d_k+1)\leq D+1$ then 
each graph  with
maximum degree $D$ has a $(d_1,\ldots,d_k)$-coloring.
Both bounds are sharp for  complete graphs.

By now, there are many papers on defective coloring and its variations such as defective list coloring and defective DP-coloring, see e.g.~\cite{Ar1,CCW1,EKKOS,HW18,KYu,LL66,SN18,OOW,VW,Woodall1,BIMR11,BIMR12,BK14,KKZ14,KKZ15,JKMSX,JKMX,KX2}. A good survey on the topic is~\cite{WoodS}.

While it is easy to check whether a graph is bipartite, i.e. is $(0,0)$-colorable, for every $(i,j)\neq (0,0)$,  the problem
 to decide whether
 a graph  has an $(i,j)$-coloring  is  NP-complete.  
Esperet, Montassier, Ochem, and Pinlou~\cite{EMOP} showed that
 even the problem of checking whether a given planar graph of girth $9$ has a $(0,1)$-coloring is NP-complete. This makes defective colorings with two colors interesting. There was a series of results on $(i,j)$-colorings of sparse graphs. A number of them
 was showing that graphs $G$ with low {\em maximum average degree}, $mad(G):=\max_{G'\subseteq G}\frac{2|E(G')|}{|V(G')|}$, are $(i,j)$-colorable.

 In particular, Borodin,  Ivanova,  Montassier, Ochem and Raspaud~\cite{BIMOR10} proved that every graph $G$ with $mad(G)<\frac{3j+4}{j+2}$ is $(0,j)$-colorable, and Borodin,  Ivanova,  Montassier  and Raspaud~\cite{BIMR12} showed  that for $j\geq 2$ every graph $G$ with $mad(G)<\frac{10j+22}{3j+9}$ is $(1,j)$-colorable. On the other hand, they presented
 non-$(1,j)$-colorable graphs with maximum average degree arbitrarily close to $\frac{14j}{4j+1}$.
  Borodin, Kostochka and Yancey~\cite{BKY13} proved 
 that every graph $G$ with $mad(G)\leq\frac{14}{5}$ is $(1,1)$-colorable, which is exact.

 A finer than $mad(G)$ measure of graph sparsity is the notion of $(a,b)$-sparse graphs. For a positive real $a$ and any real $b$, a graph $G$
 is {\em $(a,b)$-sparse} (respectively, {\em $(a,b)^*$-sparse}) if for every non-empty subgraph $G'$ of $G$, $|E(G')|\leq a|V(G')|+b$ (respectively, $|E(G')|< a|V(G')|+b$). In these terms, inequality $mad(G)< m$ means $G$ is $(m/2,0)^*$-sparse.
 Borodin and Kostochka proved in~\cite{BK11} that each
 $(\frac{6}{5},\frac{2}{5})$-sparse graph is $(0,1)$-colorable and in~\cite{BK14} that for $j\geq 2i+2$ each
 $(2-\frac{j+2}{(i+2)(j+1)},\frac{1}{j+1})$-sparse graph is $(i,j)$-colorable. Both results are tight, so the only values of $j$ for which we do not know exact bound on sparsity of a graph $G$ ensuring that
 $G$ is $(1,j)$-colorable are $2$ and $3$. In this paper, we consider $j=3$.

 Let $(a,b)$ be the lexicographical infimum of the pairs such that every $(a,b)$-sparse graph is $(1,3)$-colorable.
 By the results in~\cite{BIMR12} mentioned above,
 $\frac{13}{9}\leq a\leq \frac{42}{13}$. The  construction 
 of $G_i(j,k)$ in Section 2 of~\cite{BK14} for $j=1$ and $k=3
 $ 
 yields examples of $(\frac{19}{12},\frac{1}{4})$-sparse graphs that are not $(1,3)$-colorable, which is a better upper bound than in~\cite{BIMR12}. 
  The main result of this paper is:

\begin{theorem}
 \label{th0}
 If $|E(H)| \le \frac{14|V(H)| +5}{9}$ for every subgraph $H$ of $G$, then $G$ is $(1,3)$-colorable.
\end{theorem} 

We will prove a somewhat stronger statement  using critical graphs. A graph $G$ is {\em $(1,3)$-critical} if $G$ is not $(1,3)$-colorable, but every proper subgraph of $G$ is. 
Theorem~\ref{th0} follows from the following result on $(1,3)$-critical  graphs.

\begin{theorem}
 \label{th0critical}
 If $G$ is $(1,3)$-critical, then $14|V(G)| - 9|E(G)| \le -6$.
\end{theorem} 

Note that the bound of Theorem~\ref{th0critical} is greater than the corresponding (exact) result for  $(1,3)$-DP-coloring: there are infinitely many $(\frac{3}{2},\frac{3}{2})$-sparse graphs critical with respect to $(1,3)$-DP-coloring, see~\cite{KX2}).

\section{Setup of the proof and outline of the paper}

Our proof will use induction.
For easier induction steps,
 we will prove a stronger and more technical result.
Assume $G$ is a graph and $\bc: V(G) \to \mathbb{Z}^2$  is a map that assigns to each vertex $v$ of $G$ a pair of integers ${\bf c}(v)=(c_1(v), c_2(v))$ such that $-1 \le c_1(v) \le 1$ and $-1 \le c_2(v) \le 3$.  A \emph{${\bf c}$-coloring}  of $G$ is a mapping $\phi: V(G) \to \{1,2\}$ such that for each vertex $v$, if $\phi(v)=i$ and $V_i = \phi^{-1}(i)$, then 
$d_{G[V_i]}(v) \le c_i(v)$. In particular, if $c_i(v)=-1$, then $\phi(v) \ne i$.

We call $G$ \emph{${\bf c}$-colorable} if it admits a {${\bf c}$-coloring} and \emph{${\bf c}$-critical}, if it does not, but each its proper subgraph does. 

The ${\bf c}$-{\em potential} of 
a vertex $v$ in a graph $G$ is  $\rho_{G,{\bf c}}(v)=1+4c_1(v)+3c_2(v)$. For a subset $A$ of $V(G)$, let
\begin{equation}\label{pote}
  \rho_{G,{\bf c}}(A)=\sum_{v\in A}\rho_{G,{\bf c}}(v)-9|E(G[A])|,  
\end{equation}
  and let $$
\rho(G, {\bf c}) = \min\{\rho_{G, \bc}(A): A \subset V(G)\}.
$$

If the pair $(G, \bc)$ is clear from the context, we may write $\rho(v)$ for $\rho_{G, \bc}(v)$ and $\rho(A)$ for $\rho_{G, \bc}(A)$. Also in this case we say {\em potential} instead of ${\bf c}$-{\em potential}.
We call a vertex $v$ an $(a,b)$-vertex if $c_1(v) = a, c_2(v) = b$.
A
vertex of degree $2$ and potential $14$ is called a {\em  top  vertex}. All other vertices are {\em normal}. Let $T=T(G,{\bf c})$ denote the set of  top  vertices in $(G,{\bf c})$.
Then Theorem \ref{th0critical} follows from the following theorem.

\begin{theorem}\label{th1}
    If $G$ is ${\bf c}$-critical, then $\rho(G, \bc)\leq -6$.
\end{theorem}

The proof of Theorem \ref{th1} will be as follows.
We assume Theorem \ref{th1} is not true, and $(G, {\bf c})$ is a counterexample with 
\begin{enumerate}
\item[(A)] $|V(G)-T|$ minimum, 
\item[(B)] subject to (A), with $\rho_{G, \bc}(V(G))$ maximum.
\end{enumerate}

It is obvious that $G$ is connected and  $|V(G)| \ge 2$. 
In the next section we derive a series of properties of subsets of $V(G)$ with "low" potential. In particular,
in Corollary~\ref{geq-1} we show that every proper nonempty subset of $V(G)$ has potential at least $-1$.
In Section~\ref{pot0}, we study subsets of $V(G)$ with potential $-1$ and  $0$ and show that each such subset is either a singleton or is obtained from $V(G)$ by deleting a vertex of degree $2$. In Section~\ref{nohl}, we show that $G$ has no vertices of high potential that have low degree. In Section~\ref{dis}, we introduce and use one of our main tools, discharging. We use it to give an upper bound on $\rho_{G,{\bf c}}(V(G))$. The idea of discharging is the following. At the start, each vertex $v$ has charge
$h(v)=\rho_{G,{\bf c}}(v)-4.5d(v)$. By~\eqref{pote},
$\sum_{v\in V(G)}h(v)=\rho_{G,{\bf c}}(V(G))$. Then we change the charges of the vertices in such a way that their total sum does not change. In Section~\ref{dis}, we show that the new charge of each vertex will be non-positive, and in the final section we find a subset of vertices whose total new charge is less than $-5$. This will contradict the choice of $G$.


\section{Basic properties of minimum counter-examples}
Recall that the potential function is submodular, i.e. for every $A,B\subseteq V(G)$, 
 \begin{equation}\label{submodularity}
  \rho(A)+\rho(B)\geq \rho(A\cup B)+\rho(A\cap B).   
 \end{equation}

 Assume $S$ is a proper nonempty subset of $V(G)$ and $\phi$ is 
a $\bc$-coloring of $G[S]$. The  pair
$(G^{\phi}, \bc^{\phi})$ is defined as follows: 
For $i=1,2$, let $N_i = \{y \in V(G)-S: \phi^{-1}(i) \cap N(y) \ne \emptyset\}$.
As $G$ is connected, at least one of $N_1, N_2$ is non-empty. 
\begin{enumerate}
\item If $N_1=\emptyset$, then  
$G^{\phi}$ is obtained from $G$ by deleting $S$ and adding a vertex $y_2$ adjacent to every vertex in $N_2$, and $\bc^{\phi}$ is obtained from $\bc$ by letting $\bc^{\phi}(y_2)=(-1,0)$.
\item If $N_2=\emptyset$, then  
$G^{\phi}$ is obtained from $G$ by deleting $S$ and add a vertex $y_1$ adjacent to every vertex in $N_1$, and $\bc^{\phi}$ is obtained from $\bc$ by letting $\bc^{\phi}(y_1)=(0, -1)$.
\item If $N_1, N_2 \ne \emptyset$,  then  
$G^{\phi}$ is obtained from $G$ by deleting $S$ and add two adjacent vertices $y_1,y_2$ such that $y_i$ is adjacent to every vertex in $N_i$ for $i=1,2$,  and $\bc^{\phi}$ is obtained from $\bc$ by letting $\bc^{\phi}(y_2)=(-1,0)$ and $\bc^{\phi}(y_1)=(0,3)$.
\end{enumerate}
In each of the cases, we denote by $S^\phi$ the set of added vertices, i.e., $S^\phi=V(G^{\phi}) - V(G)$. 
We have $\rho_{G^{\phi}, \bc^{\phi}} (S^\phi) = -3$ in the first case, and $\rho_{G^{\phi}, \bc^{\phi}} (S^\phi) = -2$ otherwise.
\begin{lemma}
\label{noncolorable}
For any proper subset $S$ of $V(G)$ and any $\bc$-coloring $\phi$ of $G[S]$, $G^{\phi}$ is not $\bc^{\phi}$-colorable. Consequently if $|S| \ge 3$ or $|S|=2$ and $\rho_{G, \bc}(S) \le -3$, then $V(G^{\phi})$ has a subset $F$ with $\rho_{G^{\phi}, \bc^{\phi}}(F) \le -6$. Moreover, we can choose $F$ so that if $y_2\in S^\phi$, then $y_2\in F$ and if not then $y_1\in F$. 
\end{lemma}

\begin{proof}
If $\psi$ is a $\bc^{\phi}$-coloring of $G^{\phi}$, then the union of $\phi$ and the restriction of $\psi$ to $V(G)-S$ is a $\bc$-coloring of $G$, a contradiction.

If $|S| \ge 3$ or $|S| =2$ and $\rho_{G, \bc}(S) \le -3$, then either $|V(G^{\phi})|< |V(G)|$ or $|V(G^{\phi})|= |V(G)|$ and $\rho_{G^{\phi}, \bc^{\phi}}(V(G^{\phi}) )> \rho_{G, \bc}(V(G))$. Hence $(G^{\phi}, \bc^{\phi})$ is not a counterexample to Theorem \ref{th1}. So $V(G^{\phi})$ has a subset $F$ with $\rho_{G^{\phi}, \bc^{\phi}}(F) \le -6$. 

For the ``moreover" part, observe that if $y_2\in S^\phi-F$, then
$\rho_{G^{\phi}, \bc^{\phi}}(F\cup \{y_2\}) \le \rho_{G^{\phi}, \bc^{\phi}}(F) -3$, and if $y_2\notin S^\phi$ and $y_1 \notin F$, then $y_1\in S^\phi - F$ and
$\rho_{G^{\phi}, \bc^{\phi}}(F\cup \{y_1\}) \le \rho_{G^{\phi}, \bc^{\phi}}(F) -2$.
\end{proof}

\begin{lemma}\label{geq-2}
For every proper subset $S$ of $V(G)$, if $|S| \ge 2$, then $\rho_{G, \bc}(S)\geq -2$, and if  $|S| \ge 3$, then $\rho_{G, \bc}(S)\geq -1$. 
\end{lemma}
\begin{proof}
Assume the lemma is not true.  Choose a proper subset $S$ of $V(G)$ 
contradicting the lemma with the minimum potential.
 Since $G$ is $\bc$-critical, $G[S]$ has a $\bc$-coloring $\phi$. By Lemma \ref{noncolorable}, there is a subset $F$ of $V(G^{\phi})$ such that $\rho_{G^{\phi}, \bc^{\phi}}(F) \le -6$.

 If $y_1,y_2\in F$, or $N_2 = \emptyset$ and thus $ S^\phi = \{y_1\}$, then
\begin{equation}\label{l51}
    \rho_{G, \bc}((F \setminus S^\phi)\cup S) \leq \rho_{G^{\phi},\bc^{\phi}}(F) -\rho_{G^{\phi}, \bc^{\phi}}(S^\phi)+\rho_{G, \bc}(S) \leq {-6}-(-2)+\rho_{G, \bc}(S)={-4}+\rho_{G, \bc}(S),
\end{equation}

If $F\cap S^\phi = \{y_2\}$, then 
\begin{equation}\label{l52}
    \rho_{G, \bc}(F \setminus\{y_2\})\cup S) \leq \rho_{G^{\phi},\bc^{\phi}}(F) -\rho_{G^{\phi}, \bc^{\phi}}(y_2)+\rho_{G, \bc}(S) \leq -6+3+\rho_{G, \bc}(S)=-3+\rho_{G, \bc}(S).
\end{equation}

\noindent
This proves the first part of the statement. We now may assume $|S|\geq 3$ and $\rho_{G, \bc}(S) \leq -2$. By~\eqref{l51} and~\eqref{l52}, we know that $\rho_{G, \bc}(S) = -2$ and $F\cap S^\phi = \{y_2\}$. 
Among all such sets, let $S$ be maximum in size. 

Suppose $N_1\neq \emptyset$. Then $\rho_{G^{\phi}, \bc^{\phi}}(F \cup \{y_1\}) \le -6+10-9= -5$.  
Therefore  $$\rho_{G, \bc}((F \setminus\{y_1, y_2\})\cup S) \leq \rho_{G^{\phi},\bc^{\phi}}(F \cup \{y_1\}) -\rho_{G^{\phi}, \bc^{\phi}}(\{y_1,y_2\})+\rho_{G, \bc}(S) \leq -3+\rho_{G, \bc}(S).$$
If $(F \setminus\{y_1, y_2\})\cup S  \neq V(G)$, then this contradicts the maximality of $|S|$. Otherwise, $N_1\subset F$, we then have   $\rho_{G^{\phi}, \bc^{\phi}}(F \cup \{y_1\}) \le  -5-9 = -14$  and
$\rho_{G, \bc}((F \setminus\{y_1, y_2\})\cup S) \leq -14 + 2 + \rho_{G, \bc}(S)$, again a contradiction.

\vspace{2mm}
We now may assume $N_1 = \emptyset$ for every $\bc$-coloring on $G[S]$. 
For every $y\in N(S)$,  $-5 \le \rho_{G,\bc}(S+y)\leq -2+\rho_{\bc}(y)-9$, and thus $\rho_{\bc}(y) \ge 6$. This implies that $c_2(y) \ge 1$.

If  $\bc(y) = (-1,3)$, then $\rho_{G,\bc}(S+y)=-5$.  By the choice of $S$,  $V(G) = S\cup\{y\}$  and  hence $d(y) = 1$.
Let $x$ be the neighbor of $y$.
Let  $\bc' = \bc$, except that   $c'_2(x) = \max\{-1, c_2(x)-1\}$.
T                                                                                                   hen $\rho(G-y, \bc')\geq -5$, and by the minimality of $(G,\bc)$, $G-y$ has a $\bc'$-coloring $\psi$, which extends to a $\bc$-coloring on $G$ by letting $\psi(y) = 2$, a contradiction.

Thus $c_1(y) \ge 0$. Let $x\in  S$ be a vertex which 
has a neighbor $y \in V(G)-S$.
Let 
 $\bc'$ be obtained  from $\bc$ by letting $c'_2(x) = \max\{-1, c_2(x)-1\}$,
 and $c'_2(y) =  c_2(y)-1$.
Let $G' = G-xy$. 
If $G'$ has a $\bc'$-coloring $\psi$, then the restriction of $\psi$ to $G[S]$ is a $\bc$-coloring of $G[S]$. Hence $\psi(x) = 2$ (for otherwise $N_1 \ne \emptyset$ when we choose $\phi$ be the restriction of $\psi$ to $S$). This implies that $\psi$ is 
also a $\bc$-coloring of $G$, a contradiction.

Thus $G'$ has no $\bc'$-coloring. By the minimality of $(G, \bc)$, $(G', \bc')$ is not a counterexample to Theorem \ref{th1}.
Therefore $\rho(G', \bc') \le -6$. Let $F$ be a subset of $V(G')$ with $\rho_{G', \bc'}(F) \le -6$. 

If 
 $F\cap \{x,y\} = \emptyset$, or $\{x,y\} \subseteq F$, then $\rho_{G, \bc}(F) \le \rho_{G',\bc'}(F) \le -6$, a contradiction.
 Assume $|F \cap \{x,y\}|=1$. Then $|F| \ge 2$ and $F$ is a proper subset of $V(G')=V(G)$, and 
$\rho_{G, \bc}(F) \le \rho_{G',\bc'}(F) + 3 \le -3$, contradicting 
the first part of the statement.
\end{proof}

\begin{lemma}\label{no-3vtx}
$G$ has  no $(-1,0)$-vertices.
\end{lemma}
\begin{proof}
Suppose $v$ is a $(-1,0)$-vertex in $G$.
\begin{claim}\label{clm1}
Vertex $v$ has no degree two neighbor $u$ with $c_1 (u)= 1$.
\end{claim}

{\em Proof of claim.} Suppose $v$ has such a neighbor $u$. Let $w$ be the other neighbor of $u$.
Let $G'=G-u$ and let $\bc'$ be obtained from $\bc$ by letting $c'_1(w) = c_1(w)-1$.
If $G'$ has a $\bc'$-coloring $\phi$, then  by letting $\phi(u) = 1$ we get a $\bc$-coloring of $G$, a contradiction.

Otherwise,
by the choice of $(G,\bc)$, $G'$ contains a subset with potential at most $-6$. Let
  $S\subset V(G')$ be such a subset of maximum size.
If $w\notin S$, then $\rho_{(G, \bc)}(S)=\rho_{G', \bc'}(S)\leq -6$, and if $v\notin S$ then $\rho_{G', \bc'}(S+v)\leq \rho_{G', \bc'}(S)-3$. So, $v,w\in S$.
Therefore
$\rho_{G, \bc}(S+u)\leq \rho_{G', \bc'}(S) +\rho_{G, \bc}(u)+2\times (-9)+4\leq  -6+14-18+4 = -6$, a contradiction.
\qed

Now let $u$ be a neighbor of $v$.
\begin{claim}\label{cl1-1}
All  neighbors of $u$ apart from $v$ are  top  vertices.
\end{claim}
{\em Proof of claim.} Suppose $w$ is a 
normal neighbor of $u$ other than $v$. Let $G'=G-uw$ and let $\bc'$ be obtained from $\bc$ by letting $c'_1(w) = -1$. Then $\rho_{G',\bc'}(V(G')) > \rho_{G, \bc}(V(G))$.
Since $w$ is a normal vertex, by the choice of $(G, \bc)$, $(G', \bc')$ is not a counterexample to Theorem \ref{th1}.  

If $G'$ has a $\bc'$-coloring $\phi$, then  $\phi(w) = 2$, and $\phi(u) = 1$. Hence $\phi$ is also a $\bc$-coloring on $G$, a contradiction. Therefore,
 some $S\subset V(G')$ has $\rho_{G', \bc'}(S)\leq -6$.
 Choose a largest such  $S$.
 As in Claim~\ref{clm1}, $v,w\in S$.
If $u\in S$, then $\rho_{G, \bc}(S)\leq \rho_{G',\bc'}(S)+4\times 2 +(-9) \leq  -6+8-9$, a contradiction.
Thus $u\notin S$. 
If $S+u\neq V(G)$,  then  $$\rho_{G, \bc}(S+u)\leq
\rho_{G', \bc'}(S)+\rho(u)+2\times (-9)+8\leq
-6+14-18+8 = -2,$$ contradicting Lemma~\ref{geq-2}.
Assume $S+u = V(G)$. By Claim~\ref{clm1}, either $d_G(u) \ge 3$, or 
$d_G(u)=2$ and $c_1(u) \le 0$. Hence $$\rho_{G, \bc}(S+u)\leq 
\max\{\rho_{G', \bc'}(S)+14+3\times (-9)+8, \rho_{G', \bc'}(S)+10+2\times (-9)+8  \}=-6,$$ a contradiction.
\qed

Let $w$ be a  top  neighbor of $u$, and $z$ be the other neighbor of $w$. By Claim~\ref{clm1}, we may assume $z\neq v$.
 Let 
$G'=G-w$ and $\bc'$ be obtained from $\bc$ by letting $c'_2(z) = c_2(z)-1$. Then $(G', \bc')$ is not a counterexample to Theorem \ref{th1}. 

If $G'$ has a $\bc'$-coloring $\phi$, then  $\phi(u)=1$. So $\phi$ can be extended to a $\bc$-coloring of $G$ by letting $\phi(w)=2$, a contradiction.
Hence some $S\subset V(G')$ has $\rho_{G', \bc'}(S)\leq -6$. Let $S$ be such a set of maximum size.
Then $ v,z \in S$. If $u\in S$, 
 $\rho_{G, \bc}(S+w)\leq \rho_{G', \bc'}(S) + 3 +\rho_{G, \bc}(w)+2 \times (-9) \leq  -6+3+14-18=-7$, a contradiction.
Assume now $u\notin S$. If $u+w+S= V(G)$, then by Claim~\ref{clm1},
$\rho_{G, \bc}(S+u+w)\leq
\max\{ \rho_{G', \bc'}(S) +3 +2\times 14 +4 \times (-9), \rho_{G', \bc'}(S) +3 +14+10+3\times (-9) \}
 = -6$, a contradiction.
If $u+w+S\neq V(G)$, then $\rho_{G, \bc}(S+u+w)\leq \rho_{G', \bc'}(S) + 3 + 2\times 14+3\times (-9)\leq  -2$, contradicting Lemma~\ref{geq-2}.  This completes the proof of Lemma \ref{no-3vtx}.
\end{proof}

\begin{lemma}\label{no-2vtx}
$G$ has  no $(0,-1)$-vertices.
\end{lemma}
\begin{proof}
Suppose $v$ is a $(0,-1)$-vertex in $G$. Let $u$ be a neighbor of $v$. Let $G'=G-uv$ and let $\bc'$ be obtained from $\bc$ by letting $c'_1(u) = -1$. Then $(G', \bc')$ is not a counterexample to Theorem \ref{th1}. If $G'$ has a $\bc'$-coloring $\phi$, then $\phi(u) = 2$ and $\phi(v) = 1$. So, $\phi$ is also a $\bc$-coloring of $G$, a contradiction.
Hence some $S\subset V(G-uv)$ has $\rho_{c'}(S)\leq -6$. Let  $S$ be such a subset of  maximum size. Then $v,u\in S$. Therefore,
 $\rho_{G, \bc}(S)\leq \rho_{G', \bc'}(S)+8-9\leq -7$, a contradiction. 
\end{proof}

\begin{corollary}\label{geq-1}
For every $\emptyset\neq S\subsetneq V(G)$, $\rho_{G,\bc}(S)\geq -1$. 
\end{corollary}
\begin{proof}
    By Lemmas~\ref{no-3vtx} and \ref{no-2vtx}, every vertex has nonnegative potential.
    Suppose $S$ is a counterexample to the statement.
    By Lemma~\ref{geq-2}, we may assume $|S| = 2$ and $\rho_{G, \bc}(S) = -2$.
    Let $S = \{u,v\}$. Since both $u$ and $v$ have nonnegative potential, they are adjacent.
    By case analysis, there are only four possibilities for $\bc(u), \bc(v)$: either $\bc(u) = (-1, k)$ and $\bc(v) = (0, 3-k)$ for $k = 1,2,3$, or $\bc(u) = (1,0)$ and $\bc(v) = (1,-1)$.
    
    For the former cases, we form $\bc'$ from $\bc$ by letting $c'_2(u) = c_2(u)-1, c'_2(v) = c_2(v)-1$. 
    If $\phi$ is a $\bc'$-coloring of  $G-uv$, $\phi$ then is also a $\bc$-coloring on $G$, a contradiction. Therefore, by the choice of $G,\bc$, there is some $F\subset V(G-uv)$ with $\rho_{G-uv, \bc'}\leq -6$.
    We know that $F\cap S\neq \emptyset$.
    If $u,v\in F$, then $\rho_{G,\bc}(F)\leq \rho_{G-uv, \bc'}(F) - 9 + 6$;
    If $|F\cap S| = 1$, then $\rho_{G,\bc}(F)\leq \rho_{G-uv, \bc'}(F) +3 \leq -3$, both contradicting Lemma~\ref{geq-2}. 

    For the latter case, we form $\bc''$ from $\bc$ by letting $c''_1(u) = c_1(u)-1, c''_1(v) = c_1(v)-1$. We again may assume there is some $F\subset V(G-uv)$ with $\rho_{G-uv, \bc''}(F)\leq -6$. We may also assume that $|F\cap S| = 1$. Then we have $\rho_{G, \bc}(F)\leq -6+4 = -2$. By the submodularity of potential, we have
    $$\rho_{G,\bc}(F\cup S)\leq \rho_{G,\bc}(F) + \rho_{G,\bc}(S) - \rho_{G,\bc}(F\cap S) \leq -2-2-\min\{\rho_\bc(u), \rho_\bc(v) \} = -6, $$
    a contradiction.
\end{proof}
\section{Sets of potential $-1$ or $0$}\label{pot0}

{We say a set $ S\subseteq V(G)$ is \emph{trivial} if 
$|S|\leq 1$ or $S=V(G)$ or $S$  is obtained from   $V(G)$ by deleting a  top  vertex. Otherwise, $S$ is \emph{nontrivial}.}

Suppose that the minimum potential $j$ of a nontrivial subset of $V(G)$ is non-positive. By Corollary~\ref{geq-1},
$-1\leq j\leq 0$. Let $B$ be a largest nontrivial subset of $V(G)$ with $\rho_{G,{\bf c}}(B)=j$.

If a vertex
 $u\in V(G)\setminus B$ has at least two neighbors in $B$,
 then  $$\rho(B+u)\leq \rho(B)+\rho(u)-2\times 9 \leq 0+14-18=-4.$$ So, by Lemma~\ref{geq-2}, $B+u = V(G)$, and $u$ is a  top  vertex, thus  $B$ is trivial. Hence
\begin{equation}\label{onenei}
 \mbox{\em each $u\in V(G)\setminus B$ has at most one neighbor in $B$. }   
\end{equation}

\begin{lemma}\label{lem011}
Each vertex has positive potential.
\end{lemma}

{\bf Proof.} Suppose $\rho_{G,{\bf c}}(v)=h\leq 0$.
By Corollary~\ref{geq-1}, $h\geq -1$. 
By the definition of the potential of a vertex, 
\begin{equation}\label{B1}
\mbox{\em
 the potential of  $v$ cannot be $-1$, and if it is $0$, then  $\bc(v)=(-1,1)$. }
\end{equation}
Let $N(v)=\{u_1,\ldots,u_s\}$. If all $u_1,\ldots,u_s$ are  top  vertices, then let $G'=G-N[v]$ and let ${\bf c}'$ differ from ${\bf c}$ only in that for every vertex $w$ in $G'$, $c'_1(w)=\max\{-1, c_1(w)-|N(w) \cap N(v)|\}$.

If $G'$ has a ${\bf c}'$-coloring $\phi$, then we obtain from it  a ${\bf c}$-coloring of $G$ by letting $\phi(v)=2$ and $\phi(u_i)=1$ for all $1\leq i\leq s$. Otherwise, by the minimality of $G$, there exists $F \subseteq V(G')$ with
$\rho_{G',{\bf c}'}(F)\leq -6$. Assume $N_G(F) \cap N_G(v) =  \{ u_{i_1},u_{i_2},\ldots,u_{i_q}\}$. Then
$$\rho_{G,{\bf c}}(F \cup \{u_{i_1},u_{i_2},\ldots,u_{i_q},v\})\leq \rho_{G',{\bf c}'}(F)-q(-4)+q(-4)\leq -6,$$
a contradiction. 

Thus, $v$ has a normal neighbor, say $u$. Let $G'=G-vu$ and let ${\bf c}'$ differ from ${\bf c}$ only in that 
$c'_2(u)=c_2(u)-1$ and $c'_2(v)=c_2(v)-1$.  
If $G'$ has a ${\bf c}'$-coloring $\phi$, then $\phi$ is also a ${\bf c}$-coloring of $G$ (note that $\phi(v)=2$). Otherwise, by the minimality of $G$, there is $F \subseteq V(G')$ with
$\rho_{G',{\bf c}'}(F)\leq -6$. By Lemma~\ref{geq-2} and the definition of ${\bf c}'$, $F$ must contain both $u$ and $v$.
But then $\rho_{G,{\bf c}}(A)\leq \rho_{G',{\bf c}'}(A)-2(-3)-9\leq -9,$ a contradiction.\qed

\begin{lemma}\label{lem012}
If $B=\{v_1,v_2\}$, then $v_1v_2\in E(G)$
\end{lemma}

{\bf Proof.}
 If $B=\{v_1,v_2\}$ and $v_1v_2\notin E(G)$, then some of $v_1,v_2$ must have a nonpositive potential, contradicting
 Lemma~\ref{lem011}.
  \qed

\begin{lemma}\label{lem02}
Suppose $v\in B$  has a neighbor $u$ outside of $B$. 
 For each ${\bf c}$-coloring $\phi$ of $G[B]$ and $i\in \{1,2\}$, if $\phi(v)=i$, then $v$ has $c_i(v)$ neighbors of color $i$ in $B$.
\end{lemma}

{\bf Proof.} Suppose, for some ${\bf c}$-coloring $\phi$ of $G[B]$ and $i\in \{1,2\}$, $\phi(v)=i$ and 
\begin{equation}\label{numn}
\mbox{\em the number of neighbors of $v$ of color $i$ in $B$ is less than $c_i(v)$.}
\end{equation}

Recall that $v$ is the only neighbor of $u$ inside $B$.
Let $(G', \bc')$ be obtained from $(G^{\phi}, \bc^{\phi})$
by deleting the edge $uy_i$ and let $\bc'$ differ $\bc^{\phi}$ only in that  $c'_i(u)= c^{\phi}_i(u)-1$.

By the choice of $(G, \bc)$, $(G', \bc')$ is not a counterexample to Theorem \ref{th1}. If $(G',{\bf c}')$ is ${\bf c}'$-colorable, then together with $\phi$ we get a   ${\bf c}$-coloring of $G$. Thus there is a subset $F$ of $V(G')$ such that $\rho_{G', \bc'}(F) \le -6$.
Let $F$ be such a subset of maximum size. 

Note that $\rho_{G',{\bf c}'}(B^\phi)\geq \rho_{G,{\bf c}}(B)-3$ and $\rho_{G',{\bf c}'}(u)\geq \rho_{G,{\bf c}}(u)-4$. If $B^\phi \cup \{u\} \subset F$,  then
$$\rho_{G,{\bf c}}((F-B^\phi)\cup B)\leq \rho_{G',{\bf c}'}(F)-(-3)+j+4-9\leq -6+j-2,$$
a contradiction.
If $u \notin F$, then $$\rho_{G,{\bf c}}((F-B^\phi)\cup B)\leq
 \rho_{G', \bc'}(F)+3 \le -3,$$
  contradicting Lemma \ref{geq-2}.
 If $F \cap B^\phi = \emptyset$,  then $$\rho_{G,{\bf c}}(F)\leq
 \rho_{G', \bc'}(F)+4 \le -2,$$
  contradicting Corollary \ref{geq-1}.  
 
Assume $u \in F$, $F \cap B^\phi \ne \emptyset$ and $B^\phi \not\subseteq F$. Thus $B^\phi=\{y_1,y_2\}$ and $F$ contains exactly one of $y_1$ and $y_2$. 
By the choice of $F$, we know that $y_2 \in F$ and $y_1 \notin F$.
But then 
$$\rho_{G,\bc}((F-B^\phi)\cup B) \le \rho_{G',\bc'}(F) + j+2+4-9 \le -9,$$
a contradiction. 
\qed

Lemma~\ref{lem02} essentially says that for every $\bc$-coloring $\phi$ on $G[B]$, every vertex on the boundary of $B$ uses all its capacity in $B$ with respect to $\phi$. The statement of Lemma~\ref{lem02} also holds for all proper subsets $F\subset V(G)$ with $\rho_{G,\bc}(F)\leq 2$ 

\begin{corollary}\label{cor-1}
For any vertex $v \in B$ 
 and any color $i \in \{1,2\}$, there is a $\bc$-coloring $\phi$ of $G[B]$ such that $\phi(v)=i$.
\end{corollary}
\begin{proof} If $ c_{3-i}(v)=-1$, this holds because 
$G[B]$ has a ${\bf c}$-coloring by the minimality of $G$.

Suppose $ c_{3-i}(v)\geq 0$ and
 let weighting ${\bf c}'$ differ from ${\bf c}$ on $B$ only in that $c'_{3-i}(v)=c_{3-i}(v)-1$. Then for every nonempty $B'\subseteq B$, $\rho_{G[B'],{\bf c}'}(B)\geq j-4\ge -5$. By the minimality of
$(G,{\bf c})$, graph $G[B]$ has a ${\bf c}'$-coloring $\phi$. If $\phi(v)=3-i$, then by Lemma~\ref{lem02},
$v$ has $c_{3-i}(v)$ neighbors of color $3-i$,  contrary to the fact that $\phi$ is a $\bc'$-coloring of $G[B]$. Thus $\phi(v)=i$.
\end{proof}

Since  top  vertices are not adjacent to each other by the minimality of $G$, for each edge $uw$, at least one of $u,w$ is a normal vertex. 

\begin{lemma}\label{lem03}
Suppose $u\in V(G)-B-T$  has a neighbor $v$ in $B$. Then $u$ cannot have a neighbor $w$ in $V(G)-B-T$.
\end{lemma}

{\bf Proof.} Suppose $w\in N(u)-B-T$. Let $\phi$ be a $\bc$-coloring of $G[B]$ such that $\phi(v)=1$ (such a coloring exists by Corollary \ref{cor-1}).

Let $(G', \bc')$ be obtained from $(G^{\phi},\bc^{\phi})$ as follows: Delete edge $uw$, and obtain $\bc'$ from $\bc^{\phi}$  by decreasing each of $c^{\phi}_2(u) $ and $c^{\phi}_2(w)$ by 1.

If $(G',{\bf c}')$ has a ${\bf c}'$-coloring $\psi$, then since $u$ is adjacent to $y_1$, $\phi(y_1)=1$ and $c'_1(y_1)=0$, we know that $\psi(u)=2$. Hence the union of $\psi$ and $\phi$ is a   ${\bf c}$-coloring of $G$, a contradiction. 

Thus $G'$ is not ${\bf c}'$-colorable. 
By the minimality of $(G,{\bf c})$, $\rho_{G',{\bf c}'}(F)\leq -6$ for some subset $F$ of $V(G')$. 
Recall that and $\rho_{G',{\bf c}'}(B^\phi) \ge -3$.
Let $F'=B\cup (F-B^\phi)$.
If $F$ contains both $u,w$, then
$$\rho_{G,{\bf c}}( F')\leq \rho_{G',{\bf c}'}(F)-2(-3)+j-(-3) -9\leq -6+j,$$
a contradiction. 
If at least one of $u,w$ is not in $F$, then
$F'$ is a nontrivial subset of $V(G)$ and
$$\rho_{G,{\bf c}}( F')\leq \rho_{G',{\bf c}'}(F)-(-3)+j-(-3)\leq j,$$
contradicting the choice of $B$. 
\qed

\begin{lemma}\label{nonon-surplusnbg}
No  vertex in $V(G)-B-T$ is adjacent to $B$.
\end{lemma}
\textbf{Proof. }
Suppose $u\in V(G)-B-T$ is adjacent to $v \in B$.
Let $N$ denote the set of neighbors of $u$ in $V(G)-B$.
By Lemma~\ref{lem03}, $N\subseteq T$.
Let  $N = \{x_1,\dots, x_d\}$, and $N_2=\{w_1,\dots, w_d\}$ be the multiset of other neighbors of $x_1,\dots, x_d$. Some of $w_1,\ldots, w_d$ might coincide. 

If $|N_2\cap B|\geq 3$ (as a multiset), say $w_1,w_2,w_3 \in B$, then $\rho_{G,\bc}(B+u+\{x_1,x_2,x_3\})\leq \rho_{G,\bc}(B)+\rho_{G,\bc}(u)-3\times 4 +\rho(edge)\leq 0+14-12-9 = -7$, a contradiction. 
If $|N_2\cap B|=2$, say $w_1, w_2\in B$, then $\rho_{G,\bc}(B+u+x_1+x_2)\leq \rho_{G,\bc}(B) + \rho(u)+\rho(edge)-4\times 2 \leq 0+14-9-8=-3$. By Corollary~\ref{geq-1}, $B+u+x_1+x_2 = V(G)$, and $u$ is a $(1,3)$-vertex. 
By Corollary \ref{cor-1}, there is  a $\bc$-coloring $\phi$ of $G-u-x_1-x_2$ with $\phi(v) = 1$. We extend $\phi$ to $u,x_1,x_2$: let $\phi(x_i)\neq \phi(w_i)$ for $i = 1,2$, and $\phi(u) = 2$. Then $\phi$ is a $\bc$-coloring for $G$, a contradiction.

Thus we may assume $|N_2\cap B|\leq 1$, 
 say $w_2,\dots, w_d \notin B$, while $w_1$ might be in $B$.
And also by the maximality of $B$, we may assume $\rho(u)\geq 10$. In other words, $u$ can only be a $(1,3)$-vertex, or a $(0,3)$-vertex, or a $(1,2)$-vertex.
\vspace{2mm} \\
\textbf{Case 1: $u$ is a $(1,2)$- or $(0,3)$-vertex.}

By Corollary \ref{cor-1}, there is a $\bc$-coloring $\phi$ of $G[B]$ with $\phi(v) = 1$.

Let $(G', \bc')$ be obtained from $(G^{\phi}, \bc^{\phi})$ by 
deleting $N \cup \{u\}$, and  by decreasing $c^{\phi}_1(w_i)$ by 1 for $i = 3,\dots, d$.

If $G'$ has a $\bc'$-coloring $\psi$, then we extend $\theta = \phi \cup \psi$ to $u$
and $N$: let $\theta(u) = 2$, and $\theta(x_j) = 1$ for $j = 3,\dots, d$, $\theta(x_i)\neq \theta(w_i)$ for $i=1,2$. Then $\theta$ is a $\bc$-coloring on $G$, a contradiction.

Thus there is a subset $F$ of $V(G')$ with $\rho_{G',\bc'}(F)\leq -6$.
Let $F' = F-B^\phi+B+u +N'$, where $N'\subset N$ is the set of  top  vertices connecting $u$ and $F-B^\phi+B$. 
Let $p = |N'|$. Note that $\rho_{G',\bc'}(F \cap B^\phi) \ge \rho_{G',\bc'}(y_2) =-3$.
If $u$ is a $(0,3)$-vertex, then
$$
\rho_{G,\bc}(F')\leq \rho_{G',\bc'}(F) - \rho_{G',\bc'}(y_2)+\rho_{G,\bc}(B)+\rho_{G,\bc}(u)+\rho(edge)-4p+4p
\leq -6+3+0+10-9 = -2.
$$ 
By Corollary~\ref{geq-1}, $F' = V(G)$. So $w_1,w_2 \in F'$ and hence
\begin{multline*}
 \rho_{G,\bc}(F')\leq \rho_{G',\bc'}(F) - \rho_{G',\bc'}(y_2)+\rho_{G,\bc}(B)+\rho_\bc(u)+\rho(edge)-4d+4(d-2)  \\
 \leq -6+3+0+10-9-8 = -10,
\end{multline*}
a contradiction.

Similarly, if $u$ is a $(1,2)$-vertex, and $B^\phi \subset F$, then
$$
\rho_{G,\bc}(F')\leq \rho_{G',\bc'}(F) - \rho(B^\phi)+\rho_{G,\bc}(B)+\rho_\bc(u)+\rho(edge)-4p+4p
\leq -6+2+0+11-9 = -2.
$$
By Corollary~\ref{geq-1}, $F' = V(G)$. So $w_1,w_2 \in F'$ and hence $\rho_{G,\bc}(F')\leq  -10,$ a contradiction.
If $|B^\phi| = 2$ and $B^\phi \cap F = \{y_2\}$, then $\rho_{G,\bc}(F')\leq -1$. This is again a contradiction to the maximality of $B$ since $F'$ is nontrivial and contains $B$.
\vspace{2mm} \\
\textbf{Case 2: $u$ is a $(1,3)$-vertex.}
 
\vspace{3mm}

\noindent
\textbf{Case 2.1 $d \le 3$.} 
Let $\phi$ be  a coloring of $G[B]$ with $\phi(v) = 1$.
Let $(G',\bc')$ be obtained from $(G^{\phi}, \bc^{\phi})$ by deleting $u+N$ (without reducing the capacity of 
any vertex). Then $G'$ has a $\bc'$-coloring $\psi$. We take $\theta=\phi\cup \psi$ and extend it to $u$ by letting $\theta(u) = 2$ and $\theta(x_i)\neq \theta(w_i)$ for each $i\in [d]$. Then $\theta$ is a $\bc$-coloring on $G$, a contradiction.
\vspace{1mm}
\\
\textbf{Case 2.2 $d = 4$.} 
Let $\phi$ be  a coloring of $G[B]$ with $\phi(v) = 1$. Let $(G', \bc')$ be obtained from $(G^{\phi}, \bc^{\phi})$ by deleting $u+N$, and reducing $c^{\phi}_1(w_2)$ by 1.

If $G'$ has a $\bc'$-coloring $\psi$, then we extend $\theta = \phi\cup \psi$ to $u + N$ by letting $\theta(u) = 2$, $\theta(x_i) \neq \theta(w_i)$ for $i = 1,3,4$, and $\theta(x_2) = 1$. Then $\theta$ is a $\bc$-coloring on $G$, a contradiction.

Thus there is some $F\subset V(G')$ with $\rho_{G',\bc'}(F)\leq -6$.
Let $F' = F-B^\phi+B \subset V(G)$.
If $B^\phi\subset F$, then as $y_1 \in B^\phi$, $\rho_{G',\bc'}(B^\phi)=-2$, and hence 
$$
    \rho_{G,\bc}(F')\leq  \rho_{G',\bc'}(F) - \rho_{G',\bc'}(B^\phi) + \rho_{G,\bc}(B) + \rho_{G, \bc}(w_2)-\rho_{G',\bc'}(w_2)
    \leq -6-(-2)+j+4=j. 
$$
As $u \notin F'$, $F'$ is  a nontrivial subset of $V(G)$. This contradicts the choice of $B$. 

Thus we may assume $|B^\phi| = 2, B^\phi \cap F = \{y_2\}$.  Also by Corollary~\ref{geq-1}, $w_2\in F$, otherwise $\rho_{G,\bc}(F)\leq -6 + 2 = -4$.  Hence
$$
    \rho_{G,\bc}(F')\leq  \rho_{G',\bc'}(F) -  \rho_{G',\bc'}(y_2) + \rho_{G,\bc}(B) + \rho_{G, \bc}(w_2)-\rho_{G',\bc'}(w_2) 
    \leq -6-(-3)+j+4 \le 1+j.
$$

If $w_1 \in F'$ or $w_1 \in B$, then let 
$F''=F-B^\phi+B+\{u, x_1,x_2\}$.  In this case,
\begin{eqnarray*}
    \rho_{G,\bc}(F'')&\leq&  \rho_{G',\bc'}(F) -  \rho_{G',\bc'}(y_2) + \rho_{G,\bc}(B) + \rho_{G, \bc}(w_2)-\rho_{G',\bc'}(w_2) +\rho_{G,\bc}(\{u,x_1,x_2\})+ 5 \rho(edge) 
    \\
    &\leq& -6-(-3)+j+4 + 3 \times 14 - 5 \times 9 \le -2.
\end{eqnarray*}

If  $w_4 \in F''$, then $\rho_{G,\bc}(F''+x_4)\leq -6$, a contradiction. Thus $F''$  contradicts to    Corollary~\ref{geq-1}.
Hence, $w_1 \notin   F'$. Similarly, we can show that $w_3,w_4 \notin F'$. 

By symmetry, for each $k \in [4]$, there is some $F_k\subset V(G)-u-N$ containing $B$ and  such that  $F_k\cap N_2 = \{w_k\}$, with $\rho_{G,\bc}(F_k)\leq 1+j$.
By submodularity of potentials,
$$\rho_{G,\bc}(F_1\cup F_2)\leq \rho_{G,\bc}(F_1)+\rho_{G,\bc}(F_2)-\rho_{G,\bc}(F_1\cap F_2) \leq 1+1+2j-j \le 2.$$
(Note that $\rho_{G,\bc}(F_1\cap F_2) \ge j$, since $B\subset F_1\cap F_2$). Similarly,
$\rho_{G,\bc}(F_1\cup F_2\cup F_3)\leq 3 $. Then
$\rho_{G,\bc}(F_1\cup F_2\cup F_3 \cup \{ u,x_1,x_2,x_3\})\leq 3+\rho_\bc (u)+\rho(edge)-4\times 3 = 3+14-9-12=-4$, a contradiction to Corollary~\ref{geq-1}, since $w_4 \notin F_1\cup F_2\cup F_3\cup \{ u,x_1,x_2,x_3\}$.
\vspace{1mm}
\\
\textbf{Case 2.3 $d\geq 5$.}
Let $\phi$ be  a $\bc$-coloring of $G[B]$ with $\phi(v) = 2$.
Let $(G', \bc')$ be obtained from $(G^{\phi}, \bc^{\phi})$ by deleting $u+N$, and reducing $c^{\phi}_2(w_i)$ by 1 for $i=2,\ldots, d$. 

If $G'$ has a $\bc'$-coloring $\psi$, then we extend $\theta = \phi\cup \psi$ to $u + N$ by letting $\theta(u) = 1$, $\theta(x_1) \neq \theta(w_1)$, and $\theta(x_j) = 2$ for $j = 2,\dots, d$. Then $\theta$ is a $\bc$-coloring on $G$, a contradiction. 

Thus there is some $F\subset V(G')$ with $\rho_{G',\bc'}(F)\leq -6$.
We may assume $y_2\in F$.
Let $F_1 = F\cap \{w_1\}, F_2 = F\cap N_2 \setminus F_1$, and $p_1 = |F_1|, p_2 = |F_2|$, $N'_1$ denotes the  top  vertices connecting $u$ and $F_1$, $N'_2$ denote the  top  vertices connecting $u$ and $F_2$, and let $F' = F-B^\phi+B+u+N'_1+N'_2\subset V(G)$. Then
\begin{multline*}
    \rho_{G,\bc}(F')\leq  \rho_{G',\bc'}(F) - \rho_{G',\bc'}(B^\phi \cap F) + \rho_{G,\bc}(B) +\rho(edge)+\rho_{\bc}(u) - 4(p_1+p_2) +3p_2 \\
    \leq -6-(-3)+j-9+14-4p_1-p_2 = 2-4p_1-p_2+j.
\end{multline*}
If $p_1 = 1$, then $\rho_{G,\bc}(F')\leq -2$. By Corollary~\ref{geq-1}, $F' = V(G)$. But then since $p_2 = d-1 \geq 4$, $\rho_{G,\bc}(F')\leq 2-4-4=-6$, a contradiction.
Thus $p_1 = 0$. 

If $p_2\geq 2$,
then $\rho_{G,\bc}(F')\leq j$,   a contradiction to the choice of $B$ since $w_1\notin F'$.

If $p_2=0$, then $\rho_{G,\bc}(F-B^\phi+B) \le -6 + 3 = -3$, contrary to 
 Corollary~\ref{geq-1}. 

Suppose $p_2 = 1$, say $F_2 = \{w_2\}$.
Let $F'' = F-B^\phi+B$. Then 
$$\rho_{G,\bc}(F'')
\leq \rho_{G',\bc'}(F) - \rho(B^\phi) + \rho_{G,\bc}(B) + \rho_{G, \bc}(w_2)-\rho_{G', \bc'}(w_2)
    \leq -6-(-3)+j+3=j,
$$
contradiction to the choice of $B$ since $u\notin F''$.
\qed

Now we prove that $B$ does not exist.

Let $N$ be the set of vertices in $V(G)-B$ adjacent to $B$. By the Lemma~\ref{nonon-surplusnbg}, $N\subset T$.
Let $d = |N|$ and denote $N = \{x_1,\dots, x_d\}$, let $N' = \{w_1,\dots, w_d\} \subset V(G)-B$ be the (multi)set of the other neighbor of vertices in $N$.

Fix a $\bc$-coloring $\phi$ on $G[B]$. Define $N_i\subset N'$ so that for each $w_j\in N_i$, the other neighbor of $x_j$ in $B$ is colored $3-i$. Let $G' = G-B-N$, $\bc'_i(w_j) = \bc_i(w_j)-1$ for $w_j\in N_i$ and $i\in \{1,2\}$.
If $(G',\bc')$ has a coloring $\psi$, then we can extend $\theta = \phi\cup \psi$ to $N$ by letting $\theta(x_j) \neq \theta(y_j)$ (suppose $y_j$ is the neighbor of $x_j$ in $B$ for each $j$). Then $\theta$ is a $\bc$-coloring on $G$, a contradiction.

Thus there is some $F\subset V(G')$, with $\rho_{G',\bc'}(F)\leq -6$. Let $N_F\subset N$ be the  top  vertices connecting $B$ and $F$. Then
\begin{multline*}
    \rho_{G,\bc}(B+N_F+F)\leq \rho_{G,\bc}(B) +\rho_{G',\bc'}(F)-4|N_F|+\sum_{w\in F\cap N'}(\rho_\bc(w)-\rho_{\bc'}(w))
    \\
    \leq 
    0-6-4|N_F|+4|N_F| =-6,
\end{multline*}
 a contradiction. This yields the following.

\begin{lemma}\label{nontrivial}
Suppose $\emptyset\neq  S\subset V(G)$ is nontrivial. Then $\rho_{G,\bc}(S)\geq 1$. 
\end{lemma}

\section{$G$ has no vertices with high potential and low degree}\label{nohl}

\begin{lemma}\label{nobadvtx01}
There is no $(1,3)$-vertex in $G$ with exactly one normal neighbor and at most 4  top  neighbors.
\end{lemma}
\textbf{Proof. } Suppose $u$ is such a vertex, $v$ is its normal neighbor of $u$,
$N := \{x_1,\dots, x_d\}$ ($d = d(u)-1 \leq 4$) are the  top  neighbours of $u$, and $w_1,\dots, w_d$ (not necessarily distinct) are the other neighbors of $x_1,\dots, x_d$, respectively. 

Let $G' = G-u-N$. 
Suppose every set $W\subset V(G')$ containing $v$ has $\rho_{G',\bc}(W) \geq 2$.
Then we form $\bc'$ from $\bc$ by letting $c'_i(v) = c_i(v)-1$ for $i = 1,2$. 
By our assumption, $\rho_{G',\bc'}(A)\geq 2-3-4=-5 $ for any subset $A$ of $V(G')$. By the minimality of $G$, $G'$ has a $\bc'$-coloring $\phi$. We can  extend $\phi$ to a $\bc$-coloring of $G$ as follows:
let $\phi(x_i)\neq \phi(w_i)$ for $i\in [d]$,
let 
$\phi(u) = 2$ if there are at most 3 vertices in $N(u)$ colored $2$, or $\phi(u)=1$ if there are at least 4 
vertices in $N(u)$ colored 2, and hence at most one vertex in $N(u)$ colored 1.

Thus $v$ lies in some sets in $G'$ with 
$\bc$-potential at most $1$. Among all such sets, let $W$ be maximum in size.
\vspace{1mm}
\\
\textbf{Claim. }For every $i\in [d]$, there is some $U_i\subset V(G')$ containing $w_i$ and $v$, with $\rho_{G',\bc}(U_i)\leq 1$.
\\
\textit{Proof of Claim. } 
Suppose $(*)$: for all $U\subset V(G')$ containing $v$ and $w_1$, $\rho_{G',\bc}(U)\geq 2$.
Then $w_1\notin W$.
By Corollary \ref{cor-1},  there is a $\bc$-coloring $\phi$ of $G[W]$ with $\phi(v) = 1$.
Let $(G'', \bc'')$ be obtained from $(G^{\phi}, \bc^{\phi})$ by deleting $u + N$, and reducing 
$c^{\phi}_1(w_1)$ by 1.

If $G''$ has a $\bc''$-coloring $\psi$, then we extend $\theta = \phi\cup \psi$ to $u + N$ by letting $\theta(u) = 2$, $\theta(x_i) \neq \theta(w_i)$ for $i=2,\dots,d$, and $\theta(x_1) = 1$. Then $\theta$ is a $\bc$-coloring on $G$, a contradiction.

Thus there is some $F\subset V(G'')$ with $\rho_{G'',\bc''}(F)\leq -6$.
Let $F$ be minimum in potential and maximum in size among all such sets.
Then $\rho_{G'',\bc''}(F\cap W^\phi)\leq -2$ and $w_1\in F$.
Let $F' = F-W^\phi+W \subset V(G')$.
Then 
$$
    \rho_{G',\bc}(F')\leq \rho_{G'',\bc''}(F) +\rho_{G',\bc}(W)-\rho_{G'',\bc''}(F\cap W^\phi) +\rho_\bc(w_1)-\rho_{\bc''}(w_1) \\
    \leq 
    -6+1-(-2)+4 = 1,
$$
contradicts our assumption $(*)$. This completes the proof of the claim. 
$\bowtie$
\vspace{1mm}
\\
Let $U_1,\dots, U_d$ be as in the statement of the Claim. 
By submodularity of potential,
$$
1+1\geq \rho_{G',\bc}(U_i)+\rho_{G',\bc}(U_j) \geq \rho_{G',\bc}(U_i \cup U_j) + \rho_{G',\bc}(U_i \cap U_j)\geq \rho_{G',\bc}(U_i \cup U_j) +0,
$$
for $i,j\in [d], i\neq j$.
Thus $\rho_{G',\bc}(U_i \cup U_j)\leq 2$. We do this argument iteratively, then we get  $\rho_{G',\bc}(U:= \bigcup_{i=1}^d U_i)\leq 4$.
Then in $G$,
$$
\rho_{G,\bc}(U+u+N)\leq \rho_{G',\bc}(U)+\rho_\bc(u)+\rho(edge) - 4\times 4
\leq 4+14-9-16=-7,
$$
a
contradiction.
\qed
\begin{lemma}\label{nobadvtx011}
There is no $(1,2)$-vertex in $G$ with a normal neighbor and at most 3  top  neighbors.
\end{lemma}
\textbf{Proof. }
The proof is very similar to that of the previous lemma. So we omit it.
\qed

\begin{lemma}\label{no136vtx}
There is no $(1,3)$-vertex in $G$ with degree at most six, and all whose neighbors are  top  vertices.
\end{lemma}
\textbf{Proof. }
Suppose $u$ is such a vertex.
We may assume $d(u) = 6$, otherwise we can extend a $\bc$-coloring form $G-u$ greedily to $u$.
Let $N=N(v) = \{x_1,\dots, x_6\}$ and $N_2 = \{w_1,\dots, w_6\}$ be a multiset consists of the other neighbor of $x_i$'s.
Let $G' = G-u-N$.

Suppose for some $w_i\in N_2$, every set $U\subset V(G')$ containing $w_i$ has $\rho_{G',\bc}(U)\geq 2$.
Then by the first part of the proof of Lemma~\ref{nobadvtx01}, let $c'_j(w_i) = c_j(w_i)-1$ for $j = 1,2$, and $G'$ will have a $\bc'$-coloring $\phi$. We extend $\phi$ to $u$ and $N$:
let $\phi(x_j)\neq \phi(w_j)$ for $j\neq i$. If at this point we cannot color $u$ by $2$, then there are at least four $x_j$'s colored by $2$ already. In other words, at most one $x_j$ is colored $1$ at this point. We let $\phi(u) = 1$, and $\phi(x_i) = 2$;
If we can color $u$ by $2$ when there are only $u$ and $x_i$ uncolored, then we let $\phi(u) = 2$, $\phi(x_i) = 1$. 
In either case $\phi$ is a $\bc$-coloring on $G$, a contradiction.

Hence for each $i\in [6]$, there is some $U_i\subset V(G')$ containing $w_i$, with $\rho_{G',\bc}(U_i)\leq 1$. 
Let $U_i$ be maximum in size for each $i$.
By Lemma~\ref{nontrivial},
$\rho_{G',\bc}(U_i\cap U_j)\geq 0$ since $U_i\cap U_j$ is nontrivial and might be empty. 
By submodularity of potential, 
$$
\rho_{G',\bc}(U_i\cup U_j) \leq \rho_{G',\bc}(U_i)+\rho_{G',\bc}(U_j)-\rho_{G',\bc}(U_i\cap U_j)\leq 1+1-0=2.
$$
Let $U = \bigcup_{i=1}^6 U_i$. 
If for some $i,j\in [6]$, $|(U_i\cup U_j)\cap N_2|\geq 4$, 
say $w_1,\dots,w_4\in U_1\cup U_2$,
then since $U_i$'s are chosen maximum in size, 
$U = U_1\cup U_2\cup U_5\cup U_6$.
By iteratively applying the same submodularity argument, we get $\rho_{G',\bc}(U)\leq 4$.
Then in $G$,
$$
\rho_{G,\bc}(U+u+N)\leq \rho_{G',\bc}(U)+\rho(u)-4\times 6 \leq 4+14-24=-6,
$$
a contradiction. 
Therefore, there is some $w_i$, say $w_6$, whose corresponding $U_6$ intersects with $N_2$ at only $w_6$, and $w_6\notin U_i$ for $i\in [5]$.  
Let $U' = \bigcup_{i=1}^5 U_i, N' = \{x_1,\dots, x_5\}$. Then $\rho_{G',\bc}(U')\leq 5$. In $G$,
$$
\rho_{G,\bc}(U'+u+N')\leq \rho_{G',\bc}(U')+\rho(u)-4\times 5 \leq 5+14-20=-1,
$$
contradiction to Lemma~\ref{nontrivial} since $w_6\notin U'+u+N'$.
\qed

\begin{remark}\label{rmk2}
In the previous proof, $w_i$'s are not necessarily distinct. But if some of them coincide, it will only decrease the potential of $\rho_{G',\bc}(U)$.
\end{remark}

\begin{lemma}\label{no125vtx}
There is no $(1,2)$-vertex in $G$ with degree at most five all whose neighbors are  top  vertices.
\end{lemma}
\textbf{Proof. }
Suppose $u$ is such vertex.
Again we may assume $d(u) = 5$.
Let $N=N(v) = \{x_1,\dots, x_5\}$ and $N_2 = \{w_1,\dots, w_5\}$ be a multiset consists of the other neighbor of $x_i$'s.
Let $G' = G-u-N$.
By a similar argument as in the last proof, we may assume for each $i\in [5]$, there is some $U_i\subset V(G')$ containing $w_i$, with $\rho_{G',\bc}(U_i)\leq 1$. 
Let $U = \bigcup_{i=1}^5 U_i$.
Let $U_i$ be maximum in size for each $i$.
If for some $i,j\in [5]$, $|(U_i\cup U_j)\cap N_2|\geq 4$, 
say $w_1,\dots,w_4\in U_1\cup U_2$,
then since $U_i$'s are chosen maximum in size, 
$U = U_1\cup U_2\cup U_5$.
By iteratively applying the same submodularity argument, we get $\rho_{G',\bc}(U)\leq 3$.
Then in $G$,
$$
\rho_{G,\bc}(U+u+N)\leq \rho_{G',\bc}(U)+\rho(u)-4\times 5 \leq 3+11-20=-6,
$$ a
contradiction. 

Thus we may assume that  $w_5\notin U' := \bigcup_{i=1}^4 U_i$. Then by submodularity, $\rho_{G',\bc}(U')\leq 4$.
Let $N' = N-x_5$.
In $G$,
$$
\rho_{G,\bc}(U'+u+N')\leq \rho_{G',\bc}(U')+\rho(u)-4\times 4 \leq 4+11-16=-1,
$$
a contradiction to Lemma~\ref{nontrivial} since $w_5\notin U'+u+N'$.
\qedsymbol

\section{Discharging}\label{dis}

Let the initial charge of each vertex $v$ be  $\rho_{G,{\bf c}}(v)-4.5d(v)$. Thus the sum of charges equals $\rho_{G,\bc}(V(G))$.
The discharging rule is simple: each  top  vertex gives $2.5$ to each of its neighbors. The new charge $ch(x)$ is $0$ when $x$ is a  top  vertex and hence 
\begin{equation}\label{ch}
\rho_{G,{\bf c}}(V(G))=\sum_{v\in V(G)- T}ch(v).
\end{equation}
For $v\in V(G)-T$, let  $d_1(v)$ be the number of normal neighbors of $v$, and  $d_2(v)$ be the number of  top  neighbors of $v$. Then
\begin{equation}\label{chv}
ch(v)=1+4c_1(v)+3c_3(v)-4.5d_1(v)-2d_2(v).
\end{equation}

\begin{lemma}\label{chleq0}
For every vertex $u\in V(G)$, $ch(u)\leq 0$.
\end{lemma}
\textbf{Proof. }
Suppose for some vertex $u$, $ch(u)>0$.
Let $c_1:= c_1(u), c_2:=c_2(u), d_1:= d_1(u), d_2:=d_2(u)$.
Then we have 
$$
ch(u) = 4c_1+3c_2+1-\frac{9}{2} d_1-2d_2>0
\Rightarrow 
9d_1+4d_2\leq 8c_1+6c_2+1 \leq 27.
$$

\vspace{3mm}

\textbf{Case 1: $d_1\geq 3$.}
Then we have $d_1 = 3, d_2 = 0, c_1 = 1, c_2 = 3$.
Let $N(u) = \{ x,y,z\}$. 
Form $(G',\bc')$ from $(G,\bc)$: $G' = G-u$, $c'_2(v) = c_2(v)-1$ for $v\in N(u)$, and $\bc'$ agrees with $\bc$ everywhere else.
If $G'$ has a $\bc'$-coloring $\phi$, then by letting $\phi(u) = 2$, $\phi$ is extended to a $\bc$-coloring on $G$, a contradiction.
Thus there is some $S\subset V(G')$ with $\rho_{G',\bc'}(S)\leq -6$.
If $N(u)\subset S$, then 
\begin{multline*}
    \rho_{G,\bc}(S+u)\leq \rho_{G',\bc'}(S) + \sum_{v\in N(u)}( \rho_\bc(v)-\rho_{\bc'}(v)) +\rho_\bc(u) +3\rho(edge)
    \\
    \leq -6+3\times 3+14-9\times3 =-10,
\end{multline*}
a contradiction.

If $|S\cap N(u)|= 2$, then in $G$, 
$$\rho_{G,\bc}(S+u)\leq \rho_{G',\bc'}(S) + 2\times 3 + \rho_\bc(u)+2\rho(edge) \leq -6+6+14-9\times 2=-4.$$
But $S+u\neq V(G)$ since a neighbor of $u$ is not included, a contradiction to Corollary~\ref{geq-1}.
However, by the same corollary we must have $|S\cap N(u)|\geq 2$, otherwise $\rho_{G,\bc}(S) \leq -6+3 = -3$, again a contradiction.
\vspace{2mm}
\\
\textbf{Case 2: $d_1 = 2$.}
Then there are three possible cases:
\begin{center}
\begin{tabular}{c c | c c} 
 $d_1$ & $d_2$ & $c_1$ & $c_2$ \\ [0.5ex] 
 \hline
 $2$ & $\leq 2$ & $1$ & $3$\\ 
$2$ & $0$ & $1$ & $2$ \\
$2$ & $0$ & $0$ & $3$ \\ 
\end{tabular}    
\end{center}
In the case of the first row, $d_2\geq 1$ since by assumption $u$ is not a  top  vertex.

Let $v,w$ be the normal neighbors of $u$, $N$ be the set of  top  neighbors of $u$, and $N_2$ be the set of the other neighbor of vertices in $N$. 
By the table above, $|N|\leq 2$.

Form $(G',\bc')$: $G' = G-u-N$, $c'_2(v) = c_2(v)-1, c'_2(w) = c_2(w) - 1$, $\bc'$ agrees with $\bc$ everywhere else.
If $G'$ has a $\bc'$-coloring $\phi$, then we extend $\phi$ to $u+ N$: color every vertex in $N$ differently from its neighbor in $N_2$. If now we cannot color $u$ by $2$, then by the table, all the neighbors of $u$ are colored $2$. In this case we let $\phi(u) = 1$. 

Therefore, there is some $S\subset V(G')$ with $\rho_{G',\bc'}(S) \leq -6$. By the argument in Case 1,
$v,w\in S$.
If $d_2 = 0$, then
$$
    \rho_{G,\bc}(S+u) \leq \rho_{G',\bc'}(S) + 2\times 3+ \rho_\bc(u) + 2\rho(edge) \leq -6+6+11-18=-7,
    $$
a contradiction. 
If $d_2 > 0$, let $N'\subset N$ be the  top  vertices connecting $u$ and $S$, then when $N'\neq \emptyset$,
$$
    \rho_{G,\bc}(S+u+N') \leq \rho_{G',\bc'}(S) + 2\times 3+ \rho_\bc(u) + 2\rho(edge) - 4 \leq -6+6+14-18 - 4=-8,
    $$
contradiction. When $N' = \emptyset$, 
$\rho_{G,\bc}(S+u) \leq -4$, a contradiction to Lemma~\ref{nontrivial}.
\vspace{2mm}
\\
\textbf{Case 3: $d_1 = 1$.}
By Lemma~\ref{nobadvtx01} and Lemma~\ref{nobadvtx011}, below are all the possible cases:
\begin{center}
\begin{tabular}{c c | c c} 
 $d_1$ & $d_2$ & $c_1$ & $c_2$ \\ [0.5ex] 
 \hline
 $1$ & $1$ & $1$ & $1$\\ 
$1$ & $\leq 2$ & $0$ & $3$ \\
$1$ & $1$ & $0$ & $2$ \\ 
\end{tabular}    
\end{center}
Let $v$ be the normal neighbor of $u$.
Let $c'_2(v) = c_2(v)-1$, and $\bc'$ agrees with $\bc$ everywhere else.
Then by Lemma~\ref{nontrivial}, $G-u$ has a 
$\bc'$-coloring $\phi$. We can extend $\phi$ to a $\bc$-coloring on $G$ by letting $\phi(u) = 2$.
\vspace{2mm}
\\
\textbf{Case 4: $d_1 = 0$.}
By Lemma~\ref{no136vtx} and ~\ref{no125vtx}, the only possible cases are:
\begin{center}
\begin{tabular}{c c | c c} 
 $d_1$ & $d_2$ & $c_1$ & $c_2$ \\ [0.5ex] 
 \hline
 $0$ & $\leq 3$ & $1$ & $1$\\ 
$0$ & $2$ & $1$ & $0$ \\
$0$ & $\leq 4$ & $0$ & $3$ \\ 
$0$ & $\leq 3$ & $0$ & $2$ \\ 
$0$ & $2$ & $-1$ & $3$ \\ 
\end{tabular}    
\end{center}
Since in each case, $c_1+c_2+1\geq d_2$, we can extend a $\bc$-coloring from $G-u$ to $u$ greedily.
\qed

We have the following consequence of Lemma~\ref{chleq0} and (\ref{ch}):
\begin{corollary}\label{geq-5}
    For every vertex $v\in V(G)$, $ch(v)\geq -5$.
\end{corollary}

\section{Finishing proof of Theorem \ref{th1}}

We now finish the proof of   Theorem~\ref{th1}.
Let $G_q$ be the graph with $V(G_q)=V(G)-T$
 and $E(G_q) = E_1\cup E_q$, where $E_1=E(G-T)$ and $E_q$ is constructed as follows:  for each  top  vertex $x\in V(G)$  adjacent to vertices $u$ and $v$, we add to $E_q$ edge $x$  with endpoints $u$ and $v$. We call such $x\in E_h$ a {\em quasi-edge}.
 Note that $G_q$ may have multiple quasi-edges.
 
For a vertex $v\in V(G_q)$, let $d_1(v)$ denote the number of edges incident to $v$ in $E_1$ and let $d_2(v)$ denote the number of quasi-edges incident to $v$.

For a map $\phi: V( G_q)\rightarrow \{1,2\}$,
we define 
$$d^*_\phi (v) = |\{ uv: uv \in E_1, \phi(u)=\phi(v) \}| + \frac{1}{2} |\{ uv: uv \in E_h, \phi(u)\neq \phi(v) \}|. $$
Let $S(\phi)= \sum_{v\in V( G_q)} c_{\phi(v)}(v) - \frac{1}{2}\sum_{v\in V( G_q)} d^*_\phi (v)$, $S= \max_{\phi} S(\phi) $, and let $\psi: V( G_q)\rightarrow \{1,2\}$ be a map with $S(\psi) = S$. 


Suppose for some $u\in V( G_q)$ we have
  $c_{\psi(u)}(u)< d^*_\psi (u).$

 Let mapping $\psi_u$ differ from $\psi$ only on $u$. By Corollary~\ref{geq-5},
we have 
\begin{multline*}
    c_{\psi(u)}(u)- d^*_\psi (u) + c_{\psi_u(u)}(u)- d^*_{\psi_u} (u) = c_1(u)+c_2(u) - d_1(u)-\frac{1}{2}d_2(u) \\
    \geq \frac{1}{4}(4c_1(u) + 3c_2(u)  - \frac{9}{2}d_1(u) -2d_2(u)) \geq \frac{1}{4}(ch(u)-1) \geq  -\frac{3}{2}.
\end{multline*}

By the choice of $\psi$, 
$S(\psi_u) \leq S(\psi) $. And so 
$$S(\psi_u) - S(\psi) =  c(\psi_u(u))- d^*_{\psi_u} (u) - (c(\psi(u))- d^*_\psi (u)) \leq 0$$

Also since $c(\psi(u))- d^*_\psi (u) \equiv 0\mod{1/2} $, 
$c(\psi(u))- d^*_\psi (u) \leq -\frac{1}{2}$, and 
$$
-\frac{3}{2}\leq c(\psi_u(u))- d^*_{\psi_u} (u) + (c(\psi(u))- d^*_\psi (u))\leq -\frac{1}{2}-\frac{1}{2} = -1,
$$
we must have 
\begin{equation}\label{less2}
 \mbox{$c(\psi(u))- d^*_\psi (u)=-\frac{1}{2}$ and $c(\psi_u(u))- d^*_{\psi_u} (u) \in \{-1/2,-1\} $.}   
\end{equation}
Thus 
\begin{equation}\label{-3}
ch(u) \leq 4(c(\psi(u))- d^*_\psi (u) + c(\psi_u(u))- d^*_{\psi_u} (u)) +1 \leq 4\times (-\frac{1}{2}-\frac{1}{2}) + 1 = -3.
\end{equation}
By the choice of $G$,~\eqref{-3} and Lemma~\ref{chleq0},
\begin{equation}\label{less3}
 \mbox{there is at most one  $u$ with $c_{\psi(u)}(u)< d^*_\psi (u)$.}   
\end{equation}

We say a quasi-edge  $xy\in E( G_q)$ is \emph{$\psi$-conflicting}  if $\psi(x)\neq \psi(y)$. Let $G'$ be the spanning subgraph of $ G_q$ where $E(G')$ consists of only  $\psi$-conflicting quasi-edges. 

{\bf Case 1.} There is a vertex $u$ with $c_{\psi(u)}(u)< d^*_\psi (u)$. By~\eqref{less2},
 $ d_{G'}(u)>0$ and is an odd number. Let $C$ be the component in $G'$ containing $u$. Then there is another vertex $v\in C$ with $d_{G'}(v)$ odd.
And by~\eqref{less3}, $c(\psi(v))-d^*_\psi(v)\geq 1/2$.

Let $P$ be a $uv$-path in $G'$. 
Let $G'' = G'-E(P)$.
Add a vertex $v^*$ to $G''$ and add an edge between $v^*$ and every odd-degree vertex in $G''$. Then we can decompose $E(G''+v^*)$ into cycles. Let $\tau$ be a cyclic orientation of these cycles. 
Extend $\tau$ to $E(P)$ so that $P$ is a directed path from $v$ to $u$.

We extend $\psi$ from $ G_q$ to $G$ as follows.
If a  top  vertex $x$ does not correspond to a $\psi$-conflicting quasi-edge in $ G_q$,
then its neighbors are colored the same color $j\in \{1,2\}$; in this case
color $x$ so that $\psi(x)\neq j$. If $x$ corresponds to a $\psi$-conflicting quasi-edge,
let $y$ be the head of this quasi-edge  in $G''$;
then we color $x$ so that $\psi(x) \neq \psi(y)$. 
Since $c(\psi(u))-d^*_\psi(u) = -1/2$, $c(\psi(v))-d^*_\psi(v)\geq 1/2$, and each other normal vertex $w$ has $c(\psi(w))-d^*_\psi(w)\geq 0$, by the orientation $\tau$,
$\psi$ is a coloring on $G$, a contradiction.

{\bf Case 2.}
 $c(\psi(v))-d^*_\psi(v) \geq 0$ for every $v\in V( G_q)$. We extend $\psi$ to $G$ as in Case 1,
 with the simplification that  we do not need path $P$ and let $G'' = G'$.


\end{document}